\documentclass[reqno]{amsart}

\usepackage{enumerate}
\usepackage{csquotes}

\usepackage{amssymb}
\usepackage{amsmath}
\usepackage{amsfonts}
\usepackage{amssymb}
\usepackage{amsthm}

\usepackage[british]{babel}

\usepackage[british]{isodate}
\cleanlookdateon

\newtheorem{Theorem}{Theorem}
\newtheorem{Lemma}[Theorem]{Lemma}
\newtheorem{Remark}[Theorem]{Remark}
\newtheorem{Claim}[Theorem]{Claim}
\newtheorem{Fact}[Theorem]{Fact}

\newcommand{\cJ}{\mathcal{J}}
\newcommand{\cE}{\mathcal{E}}
\newcommand{\cC}{\mathcal{C}}
\newcommand{\cT}{\mathcal{T}}
\newcommand{\cM}{\mathcal{M}}
\newcommand{\cQ}{\mathcal{Q}}

\newcommand{\riz}{r_{i,z}}
\newcommand{\biz}{b_{i,z}}
\newcommand{\miz}{m_{i,z}}

\newcommand{\dhknp}{{H_k(n,p_1,p_2)}}

\newcommand{\conditional}{\; \middle\vert \;}

\begin{document}
\title{Jigsaw percolation on random hypergraphs}
\thanks{Accepted for publication by the Applied Probability Trust (http://www.appliedprobability.org) in the Journal of Applied Probability~54.4. The first author is supported by the National Science Foundation (NSF): DMS~1600742. The second, third, and fourth authors are supported by Austrian Science Fund (FWF): P26826 and W1230, Doctoral Program ``Discrete Mathematics''. The fourth author is also supported by European Research Council (ERC): No.~639046.}
\author[B.~Bollob\'as]{B\'ela Bollob\'as}
\address{Department of Pure Mathematics and Mathematical Statistics, University of Cambridge, Wilberforce
Road, Cambridge CB3 0WA, UK, and Department of Mathematical Sciences, University
of Memphis, Memphis, TN 38152, USA, and London Institute for Mathematical
Sciences, 35a South Street, London W1K 2XF, United Kingdom.}
\email{b.bollobas@dpmms.cam.ac.uk}
\author[O.~Cooley]{Oliver Cooley}
\author[M.~Kang]{Mihyun Kang}
\address{\hspace{-.03cm}Institute of Discrete Mathematics, Graz University of Technology, Steyrergasse~30, 8010~Graz, Austria.}\email{\{cooley,kang\}@math.tugraz.at} 
\author[C.~Koch]{Christoph Koch}
\address{Mathematics Institute, University of Warwick, Zeeman Building, CV4 7AL Coventry, United Kingdom.}\email{c.koch@warwick.ac.uk}

\begin{abstract}
The jigsaw percolation process on graphs was introduced by
Brummitt, Chatterjee, Dey, and Sivakoff~\cite{BrummittChatterjeeDeySivakoff15} as a model of collaborative solutions of
puzzles in social networks. Percolation in this process may be viewed as the joint
connectedness of two graphs on a common vertex set.

Our aim is to extend a result of Bollob\'as, Riordan, Slivken, and Smith~\cite{BollobasRiordanSlivkenSmith17} concerning
this process to hypergraphs for a variety of possible definitions of connectedness.
In particular, we determine the asymptotic order of the critical threshold probability for percolation when both hypergraphs are chosen binomially at random.
\end{abstract}

\maketitle
\vspace{-.4cm}

\section{Introduction and main result} 
Jigsaw percolation on graphs was introduced by Brummitt, Chatterjee, Dey,
and Sivakoff~\cite{BrummittChatterjeeDeySivakoff15} as a model for interactions
within a social network.
It was inspired by collaboration networks and the idea that advances are often
the result of a group of people combining their ideas, skills or knowledge
to achieve better results than would have been possible individually,
as seen in scientific research
(see~\cite{Ball14,BarabasiJeongNedaRavaszSchubergVicsek02,Newman01a,Newman01b,Tebbe11}), among other fields.

Jigsaw percolation models such collaboration as a
collective puzzle-solving process.
The premise is that each of $n$ people has a piece of a puzzle which must be combined
in a certain way to solve the puzzle.
Using the model, Brummit, Chatterjee, Dey, and Sivakoff~\cite{BrummittChatterjeeDeySivakoff15}
were able to show that some networks are better than others at solving
certain puzzles, and that this is fundamentally due to the contrasting
properties of the network. This motivates both a solid understanding of
the jigsaw process and detailed knowledge of the social network in order
to optimise performance and problem-solving. 

In the model there are two possibly overlapping sets of edges coloured red and blue
defined on a common set of vertices. (In particular, any pair of vertices may form
both a red and a blue edge at the same time.) Jigsaw percolation is a
deterministic process on {\em clusters} of vertices that evolves in discrete time.
Initially, each vertex forms its own cluster and in each subsequent time-step
two clusters merge if they are joined by at least one edge of each colour.
The process stops once no two clusters can be merged. The jigsaw process {\em percolates}
if we end in a single cluster.
In particular, the process cannot percolate if either of the graphs given by blue or red edges is not connected.

More generally, given integers $1\le r \le s$, define {\em $(r,s)$-jigsaw percolation} as follows.
Let $G_1, \dots , G_s$
be graphs on the same vertex set $V$. At each discrete time $t=0, 1, \dots $,
we have a partition of $V$ into {\em clusters}. At time $t=0$ this is the
finest partition: every vertex forms its own cluster. At time $t$,
let ${\widetilde G}_t$ be the graph whose vertices are the clusters,
with two vertices joined by an edge if the corresponding clusters are joined by an edge in
at least $r$ of the graphs $G_i$, $1\le i\le s$. 
The clusters of the jigsaw process at time $t+1$ are the unions of the clusters
that belong to the same component of ${\widetilde G}_t$. The process {\em percolates} if eventually we arrive at a single cluster.
Note that a $(1,s)$-process  percolates iff the union of the graphs $G_i$ is connected
(and in particular the $(1,1)$ process on a graph $G$ percolates iff $G$ is connected).
On the other hand, if we have $(s,s)$-percolation, then each $G_i$ must be connected.
However, the connectedness of each $G_i$ is far from sufficient
for $(s,s)$-percolation. Thus we may view percolation of the $(s,s)$ process as
a notion of the ``joint connectedness'' of the graphs $G_1,\ldots,G_s$, which requires
some interaction between them rather than merely connectedness of each $G_i$ or information
about their union.
So far, only $(2,2)$-jigsaw percolation has been considered:
in most of this paper, we shall do the same (and drop $(2,2)$ from the notation),
but we shall consider $(s,s)$-jigsaw percolation for larger $s$ in Section~\ref{sec:related}. 

Returning to the motivation, the blue graph may represent a \emph{puzzle graph} of how the pieces of the puzzle may be combined to reach a solution, while the red graph may represent the \emph{people graph}, modelling friendships between the people who hold the puzzle pieces.

Brummitt, Chatterjee, Dey, and Sivakoff~\cite{BrummittChatterjeeDeySivakoff15} studied the model when the red graph is the binomial random graph and with various deterministic possibilities for the blue graph, including a Hamilton cycle, or other connected graphs of bounded maximum degree, and provided upper and lower bounds for the percolation threshold probabilities.

Gravner and Sivakoff~\cite{GravnerSivakoff14} improved on these results for many different puzzle graphs of bounded degree, and also introduced a generalised process with redundancy parameters in the number of neighbours required for clusters to merge.

The setting in which both graphs are binomial random graphs was studied by Bollob\'as, Riordan, Slivken, and Smith~\cite{BollobasRiordanSlivkenSmith17}, who determined the asymptotic order of the threshold for percolation in terms of
the product of the two associated probabilities.

Before we state their result, we set the scene. Let $G(n,p_1)$ and $G(n,p_2)$ denote the pair of random graphs on the (common) vertex set $[n]$ (for $a\in \mathbb{N}$, we define $[a]:=\{1,\dots,a\}$), where each edge is present independently with probability $p_1$ or $p_2$ respectively. Throughout the paper any \emph{unspecified} asymptotic is with respect to $n\to\infty$ and in particular we use the phrase \emph{with high probability}, abbreviated to \emph{whp}, to mean with probability tending to $1$ as $n\to\infty$. For a probability distribution $\mathcal{P}$, we use the notation $X\sim \mathcal{P}$ to mean that $X$ is a random variable with distribution $\mathcal{P}$. The main result of Bollob\'as, Riordan, Slivken, and Smith~\cite{BollobasRiordanSlivkenSmith17} is the following.

\begin{Theorem}[\cite{BollobasRiordanSlivkenSmith17}]\label{thm:graphs}
Let $0\le p_1=p_1(n),p_2=p_2(n)\le 1$ and let $G_1\sim G(n,p_1)$ and $G_2\sim G(n,p_2)$ independently. Then there exists a constant $c>0$ such that:
\begin{enumerate}
\item \label{thm:graphs:subcrit} if $p_1p_2\le \frac{1}{cn\log n}$, then whp $(G_1,G_2)$ does not percolate;
\item \label{thm:graphs:supercrit} if $p_1p_2\ge \frac{c}{n\log n}$ and $\min\{p_1,p_2\}\ge \frac{c\log n}{n}$, then whp $(G_1,G_2)$ percolates.
\end{enumerate}
\end{Theorem}

Our main aim in this paper is to extend this result to a hypergraph setting.

\subsection{Notation and model}\label{sec:notation}
Following Erd\H{o}s,  we denote by $V^{(i)}$ the set of $i$-element subsets of a set $V$ and call its elements {\em $i$-sets}.

{\bf Hypergraphs and high-order connectedness.}  Given an integer $k\ge 2$,
a \emph{$k$-uniform hypergraph} (abbreviated to \emph{$k$-graph}) $H$ consists of a set
$V=V(H)$ of {\em vertices} and a set $E=E(H)$ of {\em edges}, where $E \subset V^{(k)}$.
(Thus for $k=2$ this defines a graph.) There are several natural possibilities
for the concept of connectedness of a hypergraph: here we define some of them.
Given an integer $1\le j \le k-1$, we say that two distinct {\em $j$-sets} $J,J'\in V^{(j)}$
are \emph{$j$-tuple-connected} in $H$, abbreviated to \emph{$j$-connected},
if there is a sequence of edges $e_1,\ldots,e_m\in E$ such that
\begin{itemize}
\item $J \subset e_1$ and $J'\subset e_m$;
\item for all $1\le i \le m-1$, some $j$-set $J_i \in V^{(j)}$ is contained in $e_i \cap e_{i+1}$.
\end{itemize}
In other words, we may walk from $J$ to $J'$ using edges which consecutively intersect
in at least $j$ vertices. Furthermore any $j$-set is always $j$-connected to itself.
(For a $j$-set which is not contained in any edge, this would not be implied by the
previous definition.) A \emph{$j$-component} of $H$ is a maximal subset of $V^{(j)}$
such that any two of its elements are $j$-connected in $H$. We say that $H$ is
\emph{$j$-connected} if all the $j$-sets of $V^{(j)}$ lie in a single $j$-component.
Note that in the case $k=2$ and $j=1$ this is simply the usual definition of connectedness (and the vertex set of a
component) of a graph.

{\bf High-order jigsaw percolation.} A \emph{double $k$-graph} $(V,E_1,E_2)$ consists of a set $E_1\subset V^{(k)}$ of \emph{red} edges and a set $E_2\subset V^{(k)}$ of \emph{blue} edges on a common set $V$ of vertices. We call the $k$-graphs $(V,E_1)$ and $(V,E_2)$ the \emph{red $k$-graph} and the \emph{blue $k$-graph} respectively.

Let $(V,E_1,E_2)$ be a double $k$-graph, let $\cJ\subset V^{(j)}$, and let $\cC_0$ be a partition of $\cJ$ into $\kappa_0:=|\cC_0|$ partition classes. (Unless stated otherwise, $\cC_0$ will be the partition in which every $j$-set forms its own partition class and so $\kappa_0=|\cJ|$.) Then \emph{$j$-jigsaw percolation} is a deterministic, discrete time process which is characterised by a sequence of partitions $\{\cC_t\}_{t=0,1,\dots}$ of $\cJ$. We write $\kappa_t:=|\cC_t|$ for the number of partition classes at time $t$ and denote the partition classes by $\cC_{t,1},\dots,\cC_{t,\kappa_t}$. The partition $\cC_{t+1}$ is obtained from $\cC_t$ as follows:
\begin{enumerate}[(I)]
\item We define an auxiliary graph ${\widetilde G}_t$ on the vertex set $[\kappa_t]$, where the edge $\{i,i'\}$ is present if and only if there are $j$-sets $J_{i,1},J_{i,2}\in\cC_{t,i}$ and $J_{i',1},J_{i',2}\in\cC_{t,i'}$ and edges $e_1\in E_1$ and $e_2\in E_2$ such that $J_{i,1}\cup J_{i',1}\subset e_1$ and $J_{i,2}\cup J_{i',2}\subset e_2$.
\item If ${\widetilde G}_t$ is an empty graph, then STOP.
\item Otherwise we set $\kappa_{t+1}$ to be the number of components of ${\widetilde G}_t$
and denote these components by $\widetilde C_1,\dots,\widetilde C_{\kappa_{t+1}}$. The partition
$\cC_{t+1}$ is then given by  
$$\cC_{t+1,i}=\bigcup_{w\in \widetilde C_i}\cC_{t,w}\qquad \textrm{for}\qquad 1\le i\le \kappa_{t+1}.$$
\item If $\kappa_{t+1}=1$, then STOP.
\item Otherwise we proceed to time $t+1$.
\end{enumerate}
If the process stopped in step (IV), we say that \emph{$\cJ$ percolates on $(V,E_1,E_2)$};
otherwise, i.e.\ if the process stopped in step (II), we say that \emph{$\cJ$
does not percolate on $(V,E_1,E_2)$}. We refer to the partition classes $\cC_{t,i}$ as \emph{clusters}.

The main results in this paper deal with the setting when
$\cJ=V^{(j)}$.
In this case we simply say that \emph{$(V,E_1,E_2)$ $j$-percolates} or \emph{$(V,E_1,E_2)$
does not $j$-percolate}. The generalised version is required in the proofs of our results.

Throughout the paper we will ignore floors and ceilings whenever this does not significantly affect the argument.

 \subsection{Main result}

In this paper we consider the \emph{random binomial double $k$-graph} $\dhknp$
on vertex set $[n]$ where every $k$-set is present as a red edge with probability $p_1$
and present as a blue edge with probability $p_2$, independently of each other
and of all other $k$-sets.

Similarly to the graph case, a necessary condition for $j$-jigsaw percolation on a
double $k$-graph is that both the red and the blue $k$-graphs are $j$-connected.
In this setting the red/blue $k$-graph is a copy of the \emph{binomial random $k$-graph}
$H_k(n,p_1)$ or $H_k(n,p_2)$ respectively, i.e.\ the vertex set is $[n]$ and each
$k$-set is present independently with probability $p_1$ or $p_2$ respectively. 
It was shown in~\cite{CooleyKangKoch15b} that the (sharp) threshold for $j$-connectedness
in $H_k(n,p)$ is
\begin{equation}\label{eq:connthresh}
p_{conn}=p_{conn}(n):=\frac{j\log n}{\binom{n}{k-j}}.
\end{equation}
As in Theorem~\ref{thm:graphs}, we impose a slightly stronger condition on $\min\{p_1,p_2\}$ to ensure the $j$-connectedness in the supercritical regime. We prove the following extension of Theorem~\ref{thm:graphs}.
\begin{Theorem}\label{thm:main}
For integers $1\le j \le k-1$,  let $0\le p_1=p_1(n),p_2=p_2(n)\le 1$ and let $H\sim\dhknp$. Then there exists a constant $c=c(k,j)>0$ such that
\begin{enumerate}
\item \label{thm:main:subcrit} if $p_1p_2\le \frac{1}{cn^{2k-2j-1}\log n}$, then whp $H$ does not $j$-percolate;
\item \label{thm:main:supercrit} if $p_1p_2\ge \frac{c}{n^{2k-2j-1}\log n}$ and $\min\{p_1,p_2\}\ge \frac{c\log n}{n^{k-j}}$, then whp $H$ $j$-percolates.
\end{enumerate}
\end{Theorem}
In other words, the threshold $\hat{p}=\hat{p}(n)$ for $j$-jigsaw percolation on $\dhknp$ in terms of the product $p=p_1p_2$ is of order
$$
\hat{p}=\Theta\left(\frac{1}{n^{2k-2j-1}\log n}\right).
$$

Somewhat surprisingly, the subcritical case (Theorem~\ref{thm:main}~\ref{thm:main:subcrit})
does not seem to be easy to prove. The corresponding subcritical case for graphs
(i.e.\ $k=2$ and $j=1$) was almost trivial, but the general case requires significantly
more involved analysis.

On the other hand, the supercritical case (Theorem~\ref{thm:main}~\ref{thm:main:supercrit})
becomes much easier; we prove a neat reduction to the graph case.

The methods that we apply in this paper can also be used, with minimal additional work,
to prove some generalisations of Theorem~\ref{thm:main} and related results.
We present these results and outlines of the proofs in Section~\ref{sec:related}.

We will prove the subcritical case of Theorem~\ref{thm:main} in Section~\ref{sec:subcrit},
and the supercritical case in Section~\ref{sec:supercrit}

\section{Subcritical regime: proof of Theorem~\ref{thm:main}~\ref{thm:main:subcrit}}\label{sec:subcrit}

Let $1\le j\le k-1$ be integers. As a first step  we develop necessary conditions
for any double $k$-graph $H$ on vertex set $V$ to $j$-percolate. Then, in the second step,
we show that if $H\sim\dhknp$ where $p_1p_2\le \frac{1}{cn^{2k-2j-1}\log n}$,
then whp it fails to satisfy the weakest of these necessary conditions.

\subsection{Necessary conditions for $j$-percolation}
The initial idea is to modify the algorithm (given by~(I)--(V) in
Section~\ref{sec:notation}) slightly, so that for each cluster it additionally
keeps track of two sets of edges, one red and one blue, which allow us to obtain
this cluster by a sequence of merges.

In this spirit we consider triples of the form $(\cJ_0,\cE_1,\cE_2)$ where
$\cJ_0\subset V^{(j)}$,  $\cE_1\subset V^{(k)}$, and $\cE_2\subset V^{(k)}$,
i.e.\ $\cE_1$ and $\cE_2$ are embedded into the complete double $k$-graph
$\left(V,V^{(k)},V^{(k)}\right)$. We call the edges in $\cE_1$ \emph{red}
and those in $\cE_2$  \emph{blue};
note that edges may be red and blue at the same time. The \emph{size} of a triple
is given by $|\cJ_0|$.  Furthermore, we call $(\cJ_0,\cE_1,\cE_2)$ \emph{internally spanned}
if $\cJ_0$ percolates on the double $k$-graph $(V,\cE_1,\cE_2)$. Note that these
triples play a crucial role for $j$-jigsaw percolation on double $k$-graphs
(comparable to \emph{internally spanned sets} in~\cite{BollobasRiordanSlivkenSmith17}).
\begin{Claim}\label{claim:bottleneck}
For every positive integer $N\le \binom{|V|}{j}/\binom{k}{j}$, if a double $k$-graph $H$ $j$-percolates, then it must contain an internally spanned triple of size $\ell$ for some $N \le \ell \le \binom{k}{j}N$.
\end{Claim}
\begin{proof}
We shall think of the edges of $H$ arriving one at a time: whenever an edge arrives, we check whether it can be used
to merge some clusters. Note that each edge $e$ that arrives can be used to merge a cluster $C$ with some others if
$e$ contains some $j$-set of $C$. Thus with the arrival of an edge, at most $\binom{k}{j}$ clusters
will be merged. Therefore, when the first cluster of size at least $N$ forms,
it was created by merging at most $\binom{k}{j}$ clusters, each of size at most $N-1$,
and therefore it has size at most $\binom{k}{j}N$.

Then the elements of our internally spanned triple consist of the set $\cJ$ of $j$-sets
in the cluster and the sets $\cE_1, \cE_2$ of red and blue edges respectively.  
Note that this may contain many more edges than necessary, but certainly $\cJ$ percolates
on $\cE_1$ and $\cE_2$, and therefore $(\cJ,\cE_1,\cE_2)$ is an internally spanned triple
of the required size.
\end{proof}

Counting all possible internally spanned triples of a given size does not seem easy.
Instead, we consider a relaxation of internally spanned triples.
To do so, we first define a strengthened notion of $j$-connectedness on $k$-graphs.
Let $(V,E)$ be a $k$-graph and let $\cJ\subset V^{(j)}$ be a collection of $j$-sets.
We say that a subset $\cJ^*\subset\cJ$ is \emph{$\cJ$-traversable} (in $(V,E)$) if
for every two distinct $j$-sets $J,J' \in\cJ^*$,
there is a sequence
of edges $e_1,\ldots,e_m\in E$ such that
\begin{itemize}
\item $J \subset e_1$ and $J'\subset e_m$;
\item for all $1\le i \le m-1$, some $j$-set $J_i \in \cJ$ is contained in $e_i \cap e_{i+1}$.
\end{itemize}
In other words, we may walk from $J$ to $J'$ using edges such that the intersection of
two consecutive edges contains at least one $j$-set from $\cJ$. Furthermore, any singleton
$\{J\}\subset\cJ$ is $\cJ$-traversable. To shorten notation, a collection
$\cJ$ is called \emph{traversable} if it is $\cJ$-traversable. Note that
$V^{(j)}$-traversable collections of $j$-sets are precisely
$j$-connected collections. Now we say that a pair $(\cJ,\cE)$ is \emph{traversable} if $\cJ$
is traversable in the $k$-graph $(V,\cE)$. Furthermore, a triple $(\cJ_0,\cE_1,\cE_2)$ is
\emph{traversable} if both $(\cJ_0,\cE_1)$ and $(\cJ_0,\cE_2)$ are traversable, i.e.\ if
$\cJ_0$ is traversable in both the red and blue $k$-graphs, $(V,\cE_1)$ and $(V,\cE_2)$ respectively.
\begin{Fact}\label{fact:trav}
Every internally spanned triple is traversable and thus contains an edge-minimal traversable triple.
\end{Fact}
Similarly to internally spanned triples, the \emph{size} of a traversable triple $(\cJ,\cE_1,\cE_2)$
is defined to be $|\cJ|$.
Let us denote the set of all edge-minimal traversable triples of size $\ell$
within the complete double $k$-graph on vertex set $V$ by $\cT_\ell$. Furthermore, we partition
$\cT_\ell$ according to the number of red and blue edges: for all integers $\ell,r,b$, we define
$$
\cT_{\ell,r,b}:=\left\{(\cJ_0,\cE_1,\cE_2)\in\cT_\ell \; \text{\large $\mid$} \; |\cE_1|=r \text{ and } |\cE_2|=b \right\}.
$$

To summarise, we have established a necessary condition for percolation of
an arbitrary double $k$-graph $H$: by Claim~\ref{claim:bottleneck} and
Fact~\ref{fact:trav} the existence of a triple $T\in \cT_{\ell,r,b}$ of size
$\log n\le \ell\le \binom{k}{j}\log n$ is necessary for $H$ to percolate.
Aiming for a first moment argument, we will therefore provide an upper bound on $| \cT_{\ell,r,b}|$.

\subsection{Edge-minimal traversable triples}
Our next goal is to establish a general upper bound on $|\cT_{\ell,r,b}|$ in an
arbitrary double hypergraph. This will be done in three steps:
firstly, we observe that for every $T\in \cT_{\ell,r,b}$, its number of edges
is bounded in terms of its size (Section~\ref{sec:edges});
secondly, we characterise every $T\in \cT_{\ell,r,b}$ by a pair of matrices,
which we call its \emph{blueprint} (Section~\ref{sec:blueprint});
and thirdly, we use these matrices to algorithmically construct a superset
of $\cT_{\ell,r,b}$ and bound its size from above (Section~\ref{sec:triples}).

\subsubsection{Number of edges.}\label{sec:edges}
Crucially, the number of edges in each edge-minimal traversable triple is bounded
as a function of its size.
\begin{Claim}\label{claim:maxEdges}
Every edge-minimal traversable triple of size $1\le\ell\le \binom{|V|}{j}$ contains at most $\ell-1$ edges of each colour. In particular, $|\cT_{\ell,r,b}|=0$ unless $0\le r,b\le \ell-1$.
\end{Claim}
\begin{proof}
By the definition of traversable triples it suffices to prove that every edge-minimal traversable pair $(\cJ,\cE)$ of size $1\le |\cJ|\le \binom{|V|}{j}$ satisfies $|\cE|\le |\cJ|-1$.

We proceed by induction over $|\cJ|$; for $|\cJ|=1$ the statement is trivial.
So let $|\cJ|\ge2$ and assume the statement holds for all $1\le i\le |\cJ|-1$.
Let $e\in\cE$ be an arbitrary edge. By edge-minimality the collection $\cJ$
is not traversable in the $k$-graph $(V,\cE\setminus\{e\})$, i.e.\ there exists
a partition $\cJ=X_1\dot{\cup}\dots\dot{\cup}X_m$ with $m\ge 2$ and
$1\le|X_1|,\dots,|X_m|\le |\cJ|-1$ such that no edge $e'\in\cE\setminus\{e\}$
contains a $j$-set from more than one partition class. Moreover, by edge-minimality
each edge in $\cE$ contains at least one $j$-set from $\cJ$, and thus we obtain
a partition of the edge set $\cE\setminus\{e\}=Y_1\dot{\cup}\dots\dot{\cup}Y_m$
by setting $Y_i:=\left\{e'\in\cE\mid \exists J\in X_i\colon J\subset e'\right\}$.
Since $(X_i,Y_i)$ is an edge-minimal traversable pair\footnote{If for some
$e^*\in Y_i$ the pair $(X_i,Y_i\setminus\{e^*\})$ was traversable, then
$(\cJ,\cE\setminus\{e^*\})$ would be traversable as well, contradicting edge-minimality.}
and $|X_i|\le |\cJ|-1$ we have by the induction hypothesis $|Y_i|\le |X_i|-1$ and
by summing over all $1\le i\le m$ we obtain $|\cE|-1\le |\cJ|-m\le |\cJ|-2$, as required.
\end{proof}

\subsubsection{Blueprints.}\label{sec:blueprint}
Next, in order to extract further structural information from edge-minimal
traversable triples, we fix an arbitrary pair $\sigma = (\sigma_j,\sigma_k)$
of total orders. Here $\sigma_j$ is an order on  $V^{(j)}$ and $\sigma_k$ on
$V^{(k)}$. Clearly, $\sigma$ induces a total order on each of the sets $\cJ_0$,
$\cE_1$, and $\cE_2$.

Given a triple $T=(\cJ_0,\cE_1,\cE_2)\in\cT_{\ell,r,b}$, we explore $\cJ_0$
via a breadth-first search (BFS) using only red edges, and colouring $j$-sets white
once they have been discovered. More precisely, we study the following
exploration process starting from the minimal element $J_{(1)}$ of $\cJ_0$
with respect to $\sigma$. Initially, we colour $J_{(1)}$ white and
make it active. In each subsequent step, there is an active
$j$-set, say $J$, and we consider all edges of $\cE_1$ containing $J$ which
have not been considered previously. For each of these edges in turn
(according to $\sigma$), we consider all the $j$-sets of $\cJ_0$ that it
contains and colour them white, if they are not yet coloured white, in the order
given by $\sigma$. We have then finished exploring $J$ and move on to the
next $j$-set which was coloured white in this process. Since $\cJ_0$ is
traversable in $(V,\cE_1)$, this induces a new total order on $\cJ_0$ which we call
the \emph{BFS-order} of $T$ (with respect to $\sigma$) and denote by
$\tau_\sigma=\tau_{\sigma}(T):=(J_{(1)},\dots,J_{(|\cJ_0|)})$.
(Note that $\tau_{\sigma}$ will in general be different from the order induced by $\sigma$ on $\cJ_0$.)

Additionally, we introduce \emph{marks} on the $j$-sets of $\cJ_0$. As we shall see,
marking a $j$-set in the blue process is similar to colouring it white in the
red process. Initially, $J_{(1)}$ is marked and all other $j$-sets are unmarked.
We then go through $\cJ_0$ according to $\tau_{\sigma}$. In the $i$-th step,
i.e.\ when $J_{(i)}$ is active,
we reveal all blue edges that contain $J_{(i)}$ and none of $\{J_{(i+1)},\dots,J_{(|\cJ_0|)}\}$,
one by one according to $\sigma$. Whenever we reveal a blue edge in this way,
we mark all the still unmarked $j$-sets in $\cJ_0$ that it contains.

The reason for colouring and marking the $j$-sets in $\cJ_0$ is that it allows us
to additionally keep track of two sets of parameters,
$\riz$ and $\biz$, for $1\le i\le |\cJ_0|$ and $0\le z\le\binom{k}{j}$. We say that
a red (or blue) edge performs \emph{$z$-duty} if we colour (respectively mark)
precisely $z$ many $j$-sets when it is revealed. Then $\riz$ denotes the number
of red $z$-duty edges that were revealed while $J_{(i)}$ was active.
Similarly, $\biz$ denotes the number of blue $z$-duty edges
that were revealed when $J_{(i)}$ was coloured white.
We encode this information in matrices $R:=(\riz)_{i,z}$ and $B:=(\biz)_{i,z}$
and call the pair $\pi_\sigma(T):=(R,B)$ the
\emph{blueprint} of $T$ with respect to $\sigma$.

For our upcoming arguments we introduce some notation for such matrices. Given a positive integer $a$ and an $a\times\left(\binom{k}{j}+1\right)$ matrix $M=(\miz)_{i,z}$ with non-negative integer entries, we define
$$
 f(M):=\sum_{i=1}^{a}\sum_{z=0}^{\binom{k}{j}} z\miz \qquad\text{and}\qquad
g(M):= \sum_{i=1}^{a}\sum_{z=0}^{\binom{k}{j}} \miz .
$$
Furthermore, for a non-negative integer $m$, we define
$$
\cM_{a,m}:=\left\{M\in\mathbb{Z}_{\ge 0}^{a\times \left(\binom{k}{j}+1\right)}\conditional f(M)=a-1 \ \ \text{ and }\ \  g(M)=m \right\}.
$$
(We shall only ever use these definitions with $a=\ell$ and $m$ either $b$ or $r$.)
The intuition behind these definitions is the following: $f$ allows us to reconstruct
the total number of white (or marked) $j$-sets (apart from $J_{(1)}$)
from the blueprint of a triple, while $g$ provides the number of red
(respectively blue) edges. More formally, this is stated in the following lemma.
\begin{Lemma}\label{lem:existBP}
Let $\ell>r,b\ge 0$ be integers and $\sigma$ be a pair of total orders on $V^{(j)}$ and $V^{(k)}$. For any $T=(\cJ_0,\cE_1,\cE_2)\in\cT_{\ell,r,b}$, the blueprint $\pi_\sigma(T)=(R,B)$ satisfies
\begin{itemize}
\item $f(R)=f(B)=|\cJ_0|-1=\ell-1;$
\item $g(R)=|\cE_1|=r$;
\item $g(B)=|\cE_2|=b$.
\end{itemize}
 In other words, $\pi_\sigma\left(\cT_{\ell,r,b}\right)\subset\cM_{\ell,r}\times\cM_{\ell,b}$.
\end{Lemma}
\begin{proof}
Let $T=(\cJ_0,\cE_1,\cE_2)$ and recall that $\cJ_0$ contains precisely $\ell-1$ many $j$-sets
excluding $J_{(1)}$ (which is already initially both white and marked) and is traversable in
both the red and blue $k$-graphs $(V,\cE_1)$ and $(V,\cE_2)$ respectively. Therefore,
each $j$-set in $\cJ_0\setminus\{J_{(1)}\}$ is coloured white by at least one red edge
and likewise receives a mark from at least one blue edge.
On the other hand, no $j$-set is coloured or marked more than once. Since $f(R)$ counts the number
of $j$-sets coloured white by red edges and $f(B)$ counts the number of marks given by blue edges, the first statement follows.

On the other hand, recall that $T$ contains $r$ red edges and $b$ blue edges. Since $T$ is edge-minimal, any edge, either red or blue, was revealed in the process and thus was counted in precisely one of the $\riz$ and one of the $\biz$. Thus the second and third statements follow.
\end{proof}

Next we give a (crude) upper bound on the number of such matrices.
\begin{Claim}\label{claim:nbMat}
Let $a>m\ge 0$ be integers. If $\cM_{a,m}\neq\emptyset$, then
$$
m\ge \binom{k}{j}^{-1}(a-1).
$$
Furthermore, there is a constant $C'=C'(k,j)>0$ (independent of $a,m$) such that
$$
|\cM_{a,m}|\le (C')^{a-1}.
$$
\end{Claim}
Note that for edge-minimal traversable triples, certainly the number of $j$-sets is larger
than the number of edges of each colour by Claim~\ref{claim:maxEdges}. Hence the condition $a>m$ is fulfilled whenever we want to apply this result in this paper.
\begin{proof}
For the first statement let $M$ be any matrix in $\cM_{a,m}$. Then
\begin{equation}\label{eq:bounda}
a-1=f(M)=\sum_{i=1}^a \sum_{z=0}^{\binom{k}{j}}z\miz \le \binom{k}{j}\sum_{i=1}^a \sum_{z=0}^{\binom{k}{j}}\miz = \binom{k}{j}m.
\end{equation}
For the second statement, first note that all entries of a matrix in
$\cM_{a,0}$ must be zero, hence $|\cM_{a,0}|\le 1$. Therefore let us
assume $m\ge 1$. Choosing an arbitrary $a\times \left(\binom{k}{j}+1\right)$
matrix $M$ with non-negative integer entries satisfying $g(M)=m$ can be
seen as having a set of $m$ indistinguishable elements and partitioning it into
$a\left(\binom{k}{j}+1\right)\stackrel{\eqref{eq:bounda}}{\le} 2m\binom{k}{j}^2$ potentially empty
distinguishable partition classes.
Thus, (using the fact that there are $\binom{t+s-1}{t-1}=\binom{t+s-1}{s}$ ways
of partitioning $s$ indistinguishable elements into $t$ distinguishable classes)
we certainly obtain the upper bound
\begin{align*}
|\cM_{a,m}|\le\binom{2\binom{k}{j}^2m+m-1}{m}&\le \left(\frac{3e\binom{k}{j}^2m}{m}\right)^{m}\le \left(9\binom{k}{j}^2\right)^{a-1},
\end{align*}
as claimed.
\end{proof}

\subsubsection{Triples obtained from blueprints.}\label{sec:triples}
Given $\ell$, $r$ and $b$, we provide an upper bound on $|\cT_{\ell,r,b}|$.
To this end, we construct a superset of $\cT_{\ell,r,b}$: we first choose a blueprint
$(R,B)\in \cM_{\ell,r}\times\cM_{\ell,b}$ and then construct all possible triples
$(\cJ_0,\cE_1,\cE_2)$ with $|\cJ_0|=\ell$, $|\cE_1|=r$, and $|\cE_2|=b$
(within the complete double $k$-graph) such that its collection of white $j$-sets $\cJ_0$
is traversable in the red $k$-graph $(V,\cE_1)$ (we do not demand that $\cJ_0$ is traversable
in the blue $k$-graph $(V,\cE_2)$, which means we will be overcounting).
More precisely, we construct $(\cJ_0,\cE_1,\cE_2)$
using the following algorithm \textbf{WR}-\textbf{B}
consisting of two phases: a white-and-red phase \textbf{WR}
followed by a blue phase \textbf{B}. Both \textbf{WR} and \textbf{B} terminate after $\ell$ steps,
or if at any point we do not have a valid choice; in the latter case,
the current instance of the procedure is discarded.

\vspace{.2cm}
\begin{tabular}{ll}
\hspace{7ex}\textbf{Algorithm WR-B}&  \\
\hspace{15ex}{\tt Input}: & $(R,B)\in \cM_{\ell,r}\times\cM_{\ell,b}$.\\
\hspace{15ex}{\tt Output}: & $(\cJ_0,\cE_1,\cE_2)\subset V^{(j)}\times V^{(k)}\times V^{(k)}$.\vspace{0.3cm}\\
\end{tabular}

\begin{tabular}{ll}
\hspace{7ex}\textbf{Phase WR} &  \\
\hspace{15ex}{\tt Input}: & $R \in \cM_{\ell,r}$.\\
\hspace{15ex}{\tt Output}: & $(\cJ_0,\cE_1)\subset V^{(j)}\times V^{(k)}$ and a total order $\vartheta$ on $\cJ_0$.\vspace{0.3cm}\\
\end{tabular}

\begin{tabular}{ll}
\hspace{7ex}\textbf{Phase B} & \\
\hspace{15ex}{\tt Input}: & $(\cJ_0,\cE_1)$ and $\vartheta$ from Phase \textbf{WR}, $B\in \cM_{\ell,b}$.\\
\hspace{15ex}{\tt Output}: & $\cE_2 \subset V^{(k)}$.\\
\end{tabular}
\vspace{.2cm}

We first run Phase \textbf{WR} as follows. Initially, we choose an arbitrary $j$-set $J_{(1)}\in V^{(j)}$ and colour it white,
meaning that $\cJ_0\leftarrow\{J_{(1)}\}$ and set $J_{(1)}\prec_\vartheta J_{(1)}$.
No red edges are chosen yet, so $\cE_1\leftarrow\emptyset$.  
Then in step $i=1,\dots, \ell$, we consider the $j$-set $J_{(i)}\in\cJ_0$ to be active.
For each $z=1,\dots,\binom{k}{j}$ and each $y=1,\dots,\riz$ we do the following: 
 	\begin{itemize}
 	\item choose an edge $e_{i,z,y}\in V^{(k)}$ containing $J_{(i)}$;
 	\item if $e_{i,z,y}\not\in\cE_1$, then $\cE_1\leftarrow\cE_1\cup e_{i,z,y}$; otherwise, discard the instance;
 	\item  for $x=1,\dots, z$
 		\begin{itemize}
 		\item we choose a $j$-set $J_{i,z,y,x}\subset e_{i,z,y}$;
 		\item if $J_{i,z,y,x}\not\in\cJ_0$, then $\cJ_0\leftarrow\cJ_0\cup\{J_{i,z,y,x}\}$ and for all $J\in\cJ_0$ set $J\prec_\vartheta J_{i,z,y,x}$; otherwise, discard the instance.
		\end{itemize} 
	\end{itemize}
If $i=\ell$, terminate $\textbf{WR}$ with output $(\cJ_0,\cE_1)$ and $\vartheta$.
Otherwise, set $J_{(i+1)}$ to be the minimal element of $\cJ_0$ with respect to $\vartheta$
such that $ J_{(i)}\prec_\vartheta J$ and $J\not\prec_\vartheta J_{(i)}$
(if this minimal element does not exist, discard the instance), and proceed with step $i+1$.
 
We observe that for any instance that is not discarded, the output satisfies
$|\cJ_0|=\ell$ and $|\cE_1|=r$, since $R\in\cM_{\ell,r}$ and in step $i$ precisely
$\sum_z z\riz$ many distinct $j$-sets are added to $\cJ_0$,
and precisely $\sum_z \riz$ many distinct $k$-sets are added to $\cE_1$.
Moreover, we note that the ordering $\vartheta$ obtained in this way is
simply the lexicographical order on $\cJ_0$ with respect to $i,z,y,x$,
i.e.\ $J_{i,z,y,x}\prec_\vartheta J_{i',z',y',x'}$ iff $\{i<i'\}$
or $\{i=i' \text{ and }z<z'\}$ or $\{i=i',\:z=z', \text{ and }y<y'\}$
or $\{i=i',\: z=z',\: y=y'\text{ and }x\le x'\}$.\footnote{Where $J_{(1)}$ is seen as $J_{0,1,1,1}$.}
This completes Phase \textbf{WR}.

It remains to embed the blue edges in a fashion which is consistent with $\vartheta$ and $B$.
Observe that it will actually be crucial for the argument, that the order $\vartheta$
is already fixed before we start embedding the blue edges. 

We next run Phase \textbf{B} as follows.  
We consider the ordered set $(\cJ_0,\prec_\vartheta)$ as the sequence $(J_{(1)},J_{(2)},\dots, J_{(\ell)})$,
and initialise $\cE_2\leftarrow\emptyset$. 
In step $i=1,\dots, \ell$, we consider the $j$-set $J_{(i)}$ to be active. For every $z=1,\dots,\binom{k}{j}$ and $y=1,\dots,\biz$
	\begin{itemize}
	\item choose an edge $e_{i,z,y}'\in V^{(k)}$ which contains $J_{(i)}$
	and another $J_{(i^*)}$ with $1\le i^*<i$;
	\item if $e_{i,z,y}'\not\in \cE_2$, then $\cE_2\leftarrow \cE_2\cup \{e_{i,z,y}'\}$;
	otherwise, discard the instance.
	\end{itemize} 
\noindent 	
If $i=\ell$, terminate $\textbf{B}$ with output $(\cJ_0,\cE_1,\cE_2)$. Otherwise, proceed with step $i+1$.

If Phase \textbf{B} terminates without being discarded, then the output of the Algorithm~\textbf{WR-B}
consists of $(\cJ_0,\cE_1,\cE_2)$, where $\cJ_0$ and $\cE_1$ are given in the output of
Phase~\textbf{WR} and $\cE_2$ is the output of Phase~\textbf{B}.

We observe that for any instance that does not get discarded we have $|\cE_2|=b$,  since $B\in \cM_{\ell,b}$ and in every step precisely $\sum_z \biz$ distinct $k$-sets are added to $\cE_2$.

Note also that in terms of marking the $j$-sets, Phase~\textbf{B} is
less restrictive than the construction of a traversable triple. More precisely, it does not actively
mark the $j$-sets, meaning that the same $j$-sets might be re-used multiple times. Furthermore
we only insist on one previously seen $j$-set $J_{(i^*)}$ being present in the edge
$e_{i,z,y}'$, where in fact there should be $z$ (previously seen but unmarked) such $j$-sets.

Now let $\cQ_{\ell,r,b}$ denote the set of all outputs of instances of \textbf{WR-B}
that did not get discarded, i.e.\ every one of our choices in both phases was valid.

\begin{Lemma}\label{lem:constr}
For all integers $\ell>r,b\ge 0$ we have $\cT_{\ell,r,b}\subset \cQ_{\ell,r,b}$; furthermore, there is a constant $C>0$ independent of $\ell,r,b$ such that
$$
|\cQ_{\ell,r,b}|\le|V|^jC^{\ell-1}\left(|V|^{k-j}\right)^{r}\left(\ell|V|^{k-j-1}\right)^{b}.
$$
\end{Lemma}
\begin{proof}
The first assertion is simple. Fix a pair $\sigma$ of total orders on $V^{(j)}$ and $V^{(k)}$.
Now consider a triple $T\in\cT_{\ell,r,b}$ and note that by
Lemma~\ref{lem:existBP} its blueprint satisfies $\pi_\sigma(T)=(R,B)$
for some $(R,B)\in \cM_{\ell,r}\times\cM_{\ell,b}$. Knowing $T$, $\sigma$,
and therefore also $\tau_\sigma(T)$, it is straightforward to provide an instance
of the above algorithm \textbf{WR}-\textbf{B} with input $(R,B)$ and output $T$
that does not get discarded.

For the upper bound on the total number of outputs let us first fix any
particular blueprint $(R,B)\in \cM_{\ell,r}\times\cM_{\ell,b}$.
For the red edges, as far as an upper bound is concerned, in each step and
for each $z$ we may choose $\riz$ elements from a set of at most $|V|^{k-j}$
with replacement and then there are at most $\binom{k}{j}^z$ ways to choose
the white $j$-sets within any red edge. (Of course these bounds are
quite crude in general, but sufficient for our result.) Since we have
at most $|V|^j$ choices for the initial white $j$-set, we have at most
\begin{align}
 |V|^j\prod_{i,z}\left(\binom{k}{j}^z |V|^{k-j}\right)^{\riz}
 &=|V|^j\binom{k}{j}^{f(R)} \left(|V|^{k-j}\right)^{g(R)}
 \stackrel{\mbox{L.~\ref{lem:existBP}}}{=}|V|^j\binom{k}{j}^{\ell-1} \left(|V|^{k-j}\right)^{r}\nonumber
\end{align}
instances which can be distinguished by their combined white and red structure.

Now consider the number of ways of choosing a blue edge together with the previously seen $j$-set $J_{(i^*)}$.
A blue edge contains at least $j+1$ already embedded vertices, for instance those in $J_{(i)}\cup J_{(i^*)}$, and there are at most $\ell$ choices for $J_{(i^*)}$. Thus we have at most $\ell |V|^{k-j-1}$ choices for this edge. Therefore the number of choices for the blue edges is at most 
\begin{align}
   \prod_{i,z}(\ell|V|^{k-j-1})^{\biz} &=(\ell|V|^{k-j-1})^{g(B)}=(\ell|V|^{k-j-1})^{b} \nonumber.
\end{align}

Finally, we have already counted the number of matrices $R\in\cM_{\ell,r}$ and $B\in\cM_{\ell,b}$ in Claim~\ref{claim:nbMat}: there is a constant $C'>0$ such that we have
   $$
   |\cM_{\ell,r}|\le (C')^{\ell-1} \qquad \text{ and } \qquad |\cM_{\ell,b}|\le (C')^{\ell-1}.
   $$
   Combining this with the previous calculations provides the desired upper bound with
   $C=(C')^2\binom{k}{j}$. 
\end{proof}

\begin{Remark}\label{rem:constr}
While the definition of (edge-minimal) traversable triples is symmetric in the
two colours red and blue, the bound in Lemma~\ref{lem:constr} is not. However,
by exchanging the roles of the two colours we also obtain 
$$
|\cQ_{\ell,r,b}|\le|V|^jC^{\ell-1}\left(|V|^{k-j}\right)^{b}\left(\ell|V|^{k-j-1}\right)^{r}.
$$
\end{Remark}

\subsection{Proof of Theorem~\ref{thm:main}~(i)}
Suppose
$$
p_1 p_2=\frac{1}{cn^{2k-2j-1}\log n}
$$ where $c>0$ is a sufficiently large constant. As mentioned earlier, it is a necessary condition for $j$-percolation of $H\sim\dhknp$ that both its red and blue $k$-graphs are $j$-connected. Therefore, by~\eqref{eq:connthresh}, we may without loss of generality assume
$$
\frac{j(k-j)!\log n}{2n^{k-j}}\le p_1,p_2\le \frac{2}{j(k-j)! c n^{k-j-1}(\log n)^2},
$$
where the upper bound follows directly from the lower bound and the assumption on the product $p:=p_1p_2$. We will prove that $\dhknp$ does not $j$-percolate by showing that there is a bottleneck in the process.

By Claim~\ref{claim:bottleneck} (for $N=\log n$), Fact~\ref{fact:trav}, and Claim~\ref{claim:maxEdges} we obtain the upper bound
 $$
 {\mathbb P}\left(H \text{ $j$-percolates}\right)\le \sum_{\ell,r,b}\:\sum_{T\in \cT_{\ell,r,b}}{\mathbb P}\left(T\subset H\right),
 $$
where (a priori) the first sum ranges over all integers $\ell$, $r$ and $b$ satisfying $\log n\le \ell \le \binom{k}{j}\log n$ and $0\le r,b\le \ell-1$. Furthermore, Lemma~\ref{lem:existBP} and Claim~\ref{claim:nbMat} imply that $\binom{k}{j}^{-1}(\ell-1)\le r,b$ whenever $\cT_{\ell,r,b}\neq\emptyset$.

 It will be convenient to split the summation into two parts, depending on whether the edge-minimal traversable triple $T=(\cJ_0,\cE_1,\cE_2)$ contains more red or blue edges.
 So let us first consider the term
 $$
 q_1:=\sum_{r\le b<\ell}\:\sum_{T\in \cT_{\ell,r,b}}{\mathbb P}\left(T\subset H\right).
 $$

Since any triple $T\in\cT_{\ell,r,b}$ contains precisely $r$ red edges and $b$ blue edges, the probability that it is contained in $H$ is
$p_1^{r}p_2^{b}$
and in particular this probability depends only on the parameters $r$ and $b$. Furthermore note that by Lemma~\ref{lem:constr} we have
$$
|\cT_{\ell,r,b}|\le n^jC^{\ell-1}\left(n^{k-j}\right)^{r}\left(\ell n^{k-j-1}\right)^{b}
$$
for some positive constant $C>0$ and thus
\begin{align}\nonumber
\sum_{T\in \cT_{\ell,r,b}}{\mathbb P}\left(T\subset H\right)&\le n^jC^{\ell-1}\left(n^{k-j}p_1\right)^{r}\left(\ell n^{k-j-1}p_2\right)^{b}\\
&\le n^j C^{\ell-1}\left(\ell n^{2k-2j-1}p\right)^{r},\label{eq:subupper}
\end{align}
since $r\le b$ and $\ell n^{k-j-1}p_2=o(1)$. Now recall that $\ell n^{2k-2j-1}p \le \binom{k}{j}/c<1$ for $c$ sufficiently large, and thus the bound in~\eqref{eq:subupper} is maximal when $r$ is minimal, i.e.\ $r= \binom{k}{j}^{-1}(\ell-1)$.
This provides the upper bound
\begin{align*}
q_1& \le \hspace{-.1cm}\sum_{r\le b<\ell} n^j\left(\frac{1}{c}\binom{k}{j}C^{\binom{k}{j}}\right)^{\frac{\ell-1}{\binom{k}{j}}} \le  \left(\binom{k}{j}\log n\right)^3n^j\left(\frac{1}{c}\binom{k}{j}C^{\binom{k}{j}}\right)^{(\log n-1)/\binom{k}{j}}\hspace{-.2cm}=o(1)
\end{align*}
for any sufficiently large constant $c$ dependent on $k$ and $j$
(since $C$ is also only dependent on $k$ and $j$).

To complete the proof consider
$$
 q_2:=\sum_{b\le r<\ell}\:\sum_{T\in \cT_{\ell,r,b}}{\mathbb P}\left(T\subset H\right).
 $$
Swapping the roles of red and blue (as indicated in Remark~\ref{rem:constr}) we can use precisely the same argument to show that $q_2=o(1)$ (using $p_1':=p_2$, $p_2':=p_1$, $\ell':=\ell$, $r':=b$ and $b':=r$). Thus
$$
{\mathbb P}\left(H \text{ $j$-percolates}\right)\le q_1+q_2=o(1),
$$
i.e.\ the binomial random double $k$-graph $\dhknp$  whp does not $j$-percolate.

\section{Supercritical regime: proof of Theorem~\ref{thm:main}~\rm{(ii)}}\label{sec:supercrit}

As we shall see, the supercritical regime is easier to prove, since we can reduce to the
graph case $k=2$ and $j=1$. In the first step we provide the reduction for
any $k\ge 3$ but $j=1$.  We then show how to obtain the result for arbitrary
pairs $(k,j)$ from the result for $(k-1,j-1)$ provided $j\ge 2$.

Given integers $1\le j \le k-1$, we define a statement $S(k,j)$ as follows.
\begin{displayquote}
$S(k,j)$: There exists a constant $c=c(k,j)>0$ and an integer $n_0=n_0(k,j)$
such that for any functions
$p_1=p_1(n)$ and $p_2=p_2(n)$ which satisfy $\frac{c\log n}{n^{k-j}} \le p_1,p_2 \le 1$
and $p_1p_2\ge \frac{c}{n^{2k-2j-1}\log n}$ for $n\ge n_0$,
the double $k$-graph $H\sim \dhknp$ $j$-percolates whp.
\end{displayquote}

Showing that $S(k,j)$ holds for all pairs of integers $1\le j<k$ proves Theorem~\ref{thm:main}~(ii).

We proceed inductively, with the base case being the result on graphs, proved in Theorem~\ref{thm:graphs}~(i).

\begin{Remark}\label{prop:graph} $S(2,1)$ holds.
\end{Remark}

We split the induction step into two parts.

\begin{Claim}\label{claim:a}
If $k\ge 3$, then $S(2,1)$ implies $S(k,1)$.
\end{Claim}
\begin{proof} Let $k\ge 3$ be an integer. Assume $S(2,1)$ holds and let $c'=c(2,1)$. Set $c=c(k,1):=3(2k)^{2k}c'$. By monotonicity, in order to show $S(k,1)$ it suffices to prove that $H\sim\dhknp$ $1$-percolates for probabilities $p_1=p_1(n)$ and $p_2=p_2(n)$ satisfying
$$
p=p_1p_2=\frac{c}{n^{2k-3}\log n}
$$
and
$$
\frac{c\log n}{n^{k-1}}\le p_1,p_2\le \frac{1}{n^{k-2}(\log n)^2},
$$
where the upper bound is an immediate consequence of the lower bound and the first condition.

We use the following construction to reduce to the graph case. Split the vertex set
$[n]$ into two disjoint sets, say $Q:=[\lfloor n/2\rfloor ]$ and $Q^*:=[n]\setminus Q$. Let
$H'$ be the double graph on $Q$ whose edges are any pair in $Q$ contained in an edge of $H$
whose remaining $k-2$ vertices are all in $Q^*$. Note that if $H'$ $1$-percolates,
then all vertices of $Q$ must lie in a single cluster $\cC$ of the final partition
$\cC_\infty$ of the $1$-jigsaw percolation process in $H$. Notice that
$H' \sim H_2(n',p_1',p_2')$ with $n':=\lfloor n/2\rfloor $ and the edge probabilities
$p_1=p_1(n')$ and $p_2=p_2(n')$ satisfy
\begin{align}\label{eq:superupper}
p_i':=1-\left(1-p_i\right)^{\binom{n/2}{k-2}}&\ge 1-\exp\left(-p_i (n/(2k-4))^{k-2}\right)\ge \frac{p_i n^{k-2}}{(2k)^k}
\end{align}
for $i\in\{1,2\}$ and sufficiently large $n$, since $p_i=o( n^{-(k-2)})$. In particular, we have
$$
p_1'p_2' \ge \frac{p_1p_2 n^{2k-4}}{(2k)^{2k}} \ge\frac{c'}{n'\log n'}
$$
and also
$$
p_i'\ge \frac{p_i n^{k-2}}{(2k)^k}\ge \frac{c\log n}{(2k)^k n}\ge \frac{c'\log n'}{n'},
$$
for $i\in\{1,2\}$. Thus, by the choice of $c'$, $H'$ does indeed $1$-percolate whp.
Similarly, reversing the roles of $Q$ and $Q^*$, i.e.\ considering only edges with
precisely $2$ vertices from $Q^*$, we deduce that whp there is a cluster
$\cC^*\in\cC_\infty$ such that $Q^*\subset \cC^*$. Even though these two events may
not be independent (at least for $k=4$) applying a union bound we can guarantee
that they happen simultaneously whp. Furthermore, there are at least two edges,
one red and one blue, present in $H'$. By definition, such edges can only exist if
there are corresponding edges (which
may not be uniquely determined) of $H$ which each contain
at least one vertex from $Q$ and $Q^*$ (since $k\ge3$). Thus the clusters
$\cC$ and $\cC^*$ must coincide and contain all vertices, i.e.\ $\cC=\cC^*=[n]$.
In other words, $H$ $1$-percolates whp.
\end{proof}

\begin{Claim}\label{claim:b}
If $k>j\ge 2$, then $S(k-1,j-1)$ implies $S(k,j)$.
\end{Claim}
\begin{proof}
Fix integers $k>j\ge 2$. Assume $S(k-1,j-1)$ holds and let $c'=c(k-1,j-1)$. Set $c:=\max\{5c', (2k)^{k}\}$.
Again, by monotonicity, in order to show $S(k,j)$ it suffices to prove that $H\sim\dhknp$ $j$-percolates  for probabilities $p_1=p_1(n)$ and $p_2=p_2(n)$ satisfying
$$
p=p_1p_2=\frac{c}{n^{2k-2j-1}\log n}\qquad \text{and}\qquad p_1,p_2\ge \frac{c\log n}{n^{k-j}}.
$$

We expose the edges of $H$ in two rounds, i.e.\ for $i=1,2$ we set
$p_i':=1-\sqrt{1-p_i}\ge p_i/2$. Note that
$p_i=2p_i'-\left(p_i'\right)^2$ and thus we obtain $H$ as the union
of two independent copies $H_{\alpha}$ and $H_{\beta}$ of
$H_k(n,p_1',p_2')$. (In particular we obtain the final
partition of $j$-jigsaw percolation on $H$ by running $j$-jigsaw percolation on
$H_{\alpha}$ and using its final partition, denoted by $\cC_{\infty}[H_{\alpha}]$, as the
initial partition for $j$-jigsaw percolation on $H_{\beta}$.)

In the following we will consider certain \emph{link double $(k-1)$-graphs}
associated to $H_{\alpha}$. They are constructed as follows.
Given a vertex $v\in[n]$ we first delete from $H_{\alpha}$ all edges (i.e.\ $k$-sets) that do not
contain the vertex $v$. Then we delete $v$ from the vertex set and replace
every remaining edge $e$ with the $(k-1)$-set $e\setminus v$.
We denote by $H_{\alpha,v}$ the resulting random double $(k-1)$-graph
on vertex set $[n]\setminus\{v\}$, and call it the {\em link double $(k-1)$-graph of $v$}. 

Now note that there is a natural bijection mapping the set of $(j-1)$-sets
(respectively $(k-1)$-sets) in $H_{\alpha,v}$ to the set of $j$-sets (respectively $k$-sets)
containing $v$ in $H_{\alpha}$. Therefore any cluster in the final partition of
$(j-1)$-jigsaw percolation on $H_{\alpha,v}$ corresponds to a collection of $j$-sets
(once we have added $v$ to each) which must be contained in a cluster of
$\cC_{\infty}[H_{\alpha}]$. Therefore, whenever $H_{\alpha,v}$ $(j-1)$-percolates,
there is a unique cluster in $\cC_{\infty}[H_{\alpha}]$ which contains all $j$-sets
which include $v$, and thus we call it {\em the corresponding cluster to $v$}.
We call a vertex $v$ \emph{good} if $H_{\alpha,v}$ $(j-1)$-percolates;
vertices that are not good are called \emph{exceptional}.
This notion is motivated by the following observation. The corresponding clusters of any two 
good vertices $u$ and $u'$ overlap in all $j$-sets containing both $u$ and $u'$, thus indeed they must coincide
(since $j\ge 2$). In other words, the final partition $\cC_{\infty}[H_{\alpha}]$
contains a cluster $\cC$ which includes every $j$-set with at least one good vertex.

Hence it remains to study $j$-sets containing only exceptional vertices. For this we observe that $H_{\alpha,v}$ is distributed as
$H_{k-1}(n',p_1',p_2')$, where $n':=n-1$ and the
probabilities $p_1'=p_1'(n')$ and
$p_2'=p_2'(n')$ satisfy
$$
p_1' p_2' \ge \frac{c}{4n^{2k-2j-1}\log n}\ge \frac{c'}{(n')^{2(k-1)-2(j-1)-1}\log (n')},
$$
$$
p_i'\ge \frac{c\log n}{2n^{k-j}}\ge \frac{c'\log (n')}{(n')^{(k-1)-(j-1)}}
$$
for $i\in\{1,2\}$ and sufficiently large $n$. Consequently, by the choice of $c'$, $H_{\alpha,v}$ $(j-1)$-percolates whp and therefore the expected number of exceptional vertices is $o(n)$, say $n/\omega$, for some function $\omega\to\infty$. Thus whp there are at most $n/\sqrt{\omega}$ exceptional vertices, by Markov's inequality.

Now we run $j$-jigsaw percolation on $H_{\beta}$ with initial partition
$\cC_{0}[H_{\beta}]:=\cC_{\infty}[H_{\alpha}]$ and show that whp $\cC_{1}[H_{\beta}]=\cC_{\infty}[H_{\beta}]=[n]$,
i.e.\ there is percolation in a single step. For this it is sufficient to
show that whp for every $j$-set $J=\{u_{(1)},\dots,u_{(j)}\}$ containing \emph{only}
exceptional vertices there are edges $e_1$ (red) and $e_2$ (blue) in
$H_{\beta}$ each containing $J$ and a good vertex $v_1$ and $v_2$ respectively,
since $J':=\{u_{(1)},\dots,u_{(j-1)},v_i\}$ satisfies $J'\in \cC_{\infty}[H_{\alpha}]$ and $J'\subset e_i$. For any
$i\in\{1,2\}$ and any such $j$-set $J$, the probability that no edge $e_i$ exists is
at most
\begin{align*}
(1-p_i')^{(1-1/\sqrt{\omega})n\binom{n-j-1}{k-j-1}} & \le \exp \left(-p_i' n^{k-j}/(2k)^{k-j}\right)\\
& \le \exp \left(-(2k)^{j-k}c\log n\right) \le n^{-(2k)^j} = o\left(n^{-j}\right),
\end{align*}
for sufficiently large $n$ and by the choice of $c$. Hence we may take the union bound over $i\in\{1,2\}$ and all (such) $j$-sets. Therefore $\dhknp$ $j$-percolates whp, completing the proof.
\end{proof}
Proposition~\ref{prop:graph} and Claims~\ref{claim:a}~\&~\ref{claim:b} imply that $S(k,j)$ holds for all pairs of integers $1\le j<k$ and this proves Theorem~\ref{thm:main}~(ii).

\section{Related models}\label{sec:related}
With some small alterations these methods can also be applied for other models, for instance line graphs or in a setting with any fixed number of colours.
\subsection{Line graphs}
We consider the following random graph $L(n,p_1,p_2)$ that has a vertex for every pair of elements from the set $[n]$, i.e.\ $V=[n]^{(2)}$, and any two vertices that intersect form a red/blue edge with probabilities $p_1$ and $p_2$, respectively, independently of each other and of all pairs of vertices; disjoint vertices cannot form an edge. Note that (graph-)jigsaw percolation on this model is closely related to $2$-jigsaw percolation on $H_3(n,p_1,p_2)$. Following the lines of our proof for $k=3$ and $j=2$ we obtain the following result.
\begin{Theorem}\label{thm:line}
Let $0\le p_1=p_1(n),p_2=p_2(n)\le 1$ and let $H\sim L(n,p_1,p_2)$. Then there is a constant $c>0$ such that
\begin{enumerate}
\item  if $p_1p_2\le \frac{1}{cn\log n}$, then whp $H$ does not $j$-percolate;
\item  if $p_1p_2\ge \frac{c}{n\log n}$ and $\min\{p_1,p_2\}\ge \frac{c\log n}{n}$, then whp $H$ $j$-percolates.
\end{enumerate}
\end{Theorem}
In other words, the threshold $\hat{p}_L=\hat{p}_L(n)$ for jigsaw percolation on $L(n,p_1,p_2)$ in terms of the product $p=p_1 p_2$ is of the order
$$
\hat{p}_L=\Theta\left(\frac{1}{n\log n}\right).
$$
Indeed the proof for the subcritical regime will be simplified since there cannot
be any edges doing multiple duty (neither in red nor in blue).
For the supercritical regime one reduction step is enough, since the link graph
of a vertex in the line graph is a binomial random double graph and thus
Theorem~\ref{thm:graphs} applies.

\subsection{More colours}
Returning to the original motivation of jigsaw percolation, one might ask whether
the social network is able to collectively solve multiple puzzles simultaneously.
In this spirit we define a binomial random $s$-fold $k$-graph $H_k(n,p_1,\dots,p_s)$,
in which the vertex set is $[n]$ and any $k$-set forms an $i$-edge with probability
$p_i$ independently for all $1\le i\le s$ and all other $k$-sets.

Now in the process of $j$-jigsaw percolation, clusters merge if for each colour
there is at least one edge connecting them. Hence, based on the same intuition,
we analyse internally spanned $(s+1)$-tuples (the s-coloured analogue of internally
spanned triples) and observe that the sequence of white $j$-sets is already determined
by the set of edges of the first colour, say red. Now any further colour essentially
behaves like blue and in particular independently of the other colours, given the
sequence of white $j$-sets. With this intuition we obtain the following generalisation
of Theorem~\ref{thm:main}.
\begin{Theorem}\label{thm:multi}
For integers $1\le j < k$ and $s\ge 2$ let $0\le p_1=p_1(n)\le \dots \le p_s=p_s(n)\le 1$ and let $H\sim H_k(n,p_1,\dots,p_s)$. Then there is a constant $c=c(s,k,j)>0$ such that
\begin{enumerate}
\item if $p_1\le \frac{\log n}{cn^{k-j}}$ or for any $2\le r \le s$ we have $\prod_{i=1}^{r}p_i\le \frac{1}{cn^{r(k-j-1)+1}(\log n)^{r-1}}$, then whp $H$ does not $j$-percolate;
\item if $p_1 \ge \frac{c\log n}{n^{k-j}}$ and for each $2\le r \le s$ we have $\prod_{i=1}^{r}p_i\ge \frac{c}{n^{r(k-j-1)+1}(\log n)^{r-1}}$,
then whp $H$ $j$-percolates.  
\end{enumerate}
\end{Theorem}
In other words, the threshold $\hat{p}_s=\hat{p}_s(n)$ for $j$-jigsaw percolation on the $s$-fold $k$-graph $H_k(n,p_1,\dots,p_s)$ in terms of $p=\prod_{i=1}^{s}p_i$ is of order
$$
\hat{p}_s=\Theta\left(\frac{1}{n^{s(k-j-1)+1}(\log n)^{s-1}}\right).
$$
The condition is stated for $2\le r\le s$ since percolation of any subset of $r$ colours is
necessary for the $s$-fold $k$-graph to percolate. The condition on $p_1$ corresponds to the case
$r=1$; the $j$-percolation process for one colour simply merges $j$-components into one cluster,
thus it is important whether any of the $p_i$ are below the $j$-connectedness threshold
given in~\eqref{eq:connthresh}.

{\bf Outline of proof.} It turns out that there is a minor technical obstacle when determining the upper bound on the probabilities $p_i$ for all $i\in[s]$ from some necessary conditions for $j$-percolation of an $s$-fold $k$-graph. As before the $k$-graph of (any) colour $i\in[s]$ has to be $j$-connected, i.e.\ by~\eqref{eq:connthresh} we may assume
$$
p_i=\Omega\left(\frac{\log n}{n^{k-j}}\right).
$$
However, this alone will not yield useful upper bounds. Instead we observe
that additionally, for any proper subset $I\subsetneq\{1,\dots,s\}$ of size
at least two, it is necessary that the $|I|$-fold $k$-graph $H_k(n,(p_i)_{i\in I})$
$j$-percolates. Theorem~\ref{thm:multi} is then proved by induction over $s$
with Theorem~\ref{thm:main}  providing the base case. Hence, assume
$\prod_{j\in[s]: j\neq i}p_j=\Omega(\hat{p}_{s-1})$ for all $i\in[s]$.
Additionally, by monotonicity, we may also assume $\prod_{j\in[s]}p_j=\Theta(\hat{p}_s)$
and thus obtain the following upper bounds on the probabilities $p_i$:
 $$
 p_i=\frac{\prod_{j\in[s]}p_j}{\prod_{j\in[s]: j\neq i}p_j}= O\left(\frac{\hat{p}_s}{\hat{p}_{s-1}}\right)=O\left(\frac{1}{n^{k-j-1}\log n}\right).
 $$
Even though this bound is slightly weaker than in the two-colour case
(where we had another factor of $1/\log n$), it turns out to be sufficient
for adapting our proof. More precisely, it is used in the following two arguments:
\begin{itemize}
\item in the subcritical regime, in order to derive~\eqref{eq:subupper}.
There we only need that $\ell n^{k-j-1}p_i\le 1$ for all $\ell\le \binom{k}{j}\log n$
and all colours $i\in [s]$. This holds if the constant $c(s,k,j)$ is chosen
sufficiently large compared to $c(s-1,k,j)$.
\item in the supercritical regime, for the last estimate in~\eqref{eq:superupper}.
The argument doesn't change at all, since having a single $(1/\log n)$-factor is already enough here.
\end{itemize}
Apart from these occurrences, the upper bound is only used to reprove the statement
(in the supercritical regime) for graphs~\cite{BollobasRiordanSlivkenSmith17}
in the setting of multiple colours.
Some of the asymptotic approximations used in the proof of the $2$-colour graph case
are no longer useful with only the weaker upper bound.
However, these technical issues can be dealt with and the details can be found in~\cite{Gutierrez}.

\section{Concluding remarks}

Theorem~\ref{thm:main} holds for a large enough constant $c$, which we made no attempt to optimise (nor was any such attempt made for graphs in~\cite{BollobasRiordanSlivkenSmith17}). It would be interesting to know the exact threshold, and in particular whether it is sharp, i.e. the upper and lower thresholds are asymptotically the same.

Note also that in the supercritical regime there was an extra condition on $\min \{p_1,p_2\}$ which contained a factor of $c$.
However, there is no intrinsic reason why this $c$ should be the same as the $c$ in the product. Indeed, the reason for this
condition is to ensure that each hypergraph is $j$-connected whp, but as mentioned in the introduction, the asymptotic
threshold for this was determined in~\cite{CooleyKangKoch15b} to be $\frac{j(k-j)!\log n}{n^{k-j}}$.
It is therefore natural to conjecture that this condition can be replaced by $\min \{p_1,p_2\} \ge \frac{c'\log n}{n^{k-j}}$
for any $c'> j(k-j)!$. Indeed, even if, say $p_1=\frac{j(k-j)!\log n}{n^{k-j}}$, there is a certain probability, bounded away
from $0$ and $1$, that the red $k$-graph is $j$-connected. It is then natural to conjecture that, conditioned on it being
$j$-connected (note that whp the blue $k$-graph will also be $j$-connected), whp the double $k$-graph $\dhknp$ $j$-percolates.

\section{Acknowledgements}
We would like to thank the anonymous referees for their comments and suggestions,
which improved the quality of the paper.
\

\bibliographystyle{amsplain}
\bibliography{References}

\
\end{document}